\documentclass[12pt]{amsart}

\usepackage{amsmath, amssymb, amscd, newlfont}
\usepackage{comment}

\newcommand{\cInd}{{\text {\rm{c-Ind}}}}

\newcommand{\bbA}{{\mathbb{A}}}

\newcommand{\bbQ}{{\mathbb{Q}}}

\newcommand{\bbR}{{\mathbb{R}}}
\newcommand{\bbZ}{{\mathbb{Z}}}

\newcommand{\Hom}{{\mathrm{Hom}}}

\newcommand{\supp}{{\mathrm{supp}}}

\newcommand{\Ind}{{\mathrm{Ind}}}

\newcommand{\wit}{\widetilde}

\newcommand{\calO}{{\mathcal{O}}}

\usepackage[mathscr]{eucal}
\def \ScptA{\mathcal A}
\def \ScptB{\mathcal B}

\font\sans=cmss10

\numberwithin{equation}{section}

\newtheorem{Prop}[equation]{Proposition}
\newtheorem{Lem}[equation]{Lemma}
\newtheorem{Def}[equation]{Definition}
\newtheorem{Thm}[equation] {Theorem}
\newtheorem{Cor}[equation]{Corollary}

\newtheorem{Assumptions}[equation]{Assumptions}

\title
[
On Existence of Generic Cusp Forms
]
{On Existence of Generic Cusp Forms on Semisimple Algebraic Groups}

\author{Allen Moy and Goran Mui\'c}

\address{ Department of Mathematics,
The Hong--Kong University of Science and Technology,
Clear Water Bay, Hong Kong}
\email{amoy@ust.hk}

\address{ Department of Mathematics,
University of Zagreb,
Bijeni\v cka 30, 10000 Zagreb,
Croatia}
\email{gmuic@math.hr}

\thanks{The 1st author acknowledges Hong Kong Research Grants Council
  grant CERG {\#}603813, and the 2nd  author acknowledges Croatian Science Foundation grant no. 9364.}

\begin{document}

\begin{abstract}
  In this paper we discuss the existence of certain classes of cuspidal automorphic representations having
  non--zero Fourier coefficients for general semisimple algebraic group $G$ defined over a number field $k$
  such that its Archimedean group $G_\infty$ is not compact. When $G$ is quasi--split over $k$, we obtain a result 
  on existence of generic cuspidal automorphic representations which generalize a result of Vign\' eras, Henniart, and Shahidi. We also discuss
  the existence of  cuspidal automorphic forms with non--zero Fourier coefficients for congruence of subgroups of $G_\infty$. 
  \end{abstract}
\maketitle

\section{Introduction}

Possibly degenerate Fourier coefficients of automorphic cuspidal forms are important for the theory of automorphic
$L$--functions (\cite{Sh}, \cite{grs}, \cite{soudry},\cite{Li}).
Recent classification of discrete global spectrum for classical groups due to 
Arthur \cite{arthur} can not be used directly to study Fourier coefficients of cuspidal automorphic forms. 
The goal of the present paper is to adjust  methods of compactly
supported Poincar\' e series as developed in \cite{Muic1} in order to show existence of various types of cuspidal automorphic
forms with non--zero Fourier coefficients for a  general semisimple algebraic group $G$ over a number field $k$. We warn the
reader that compactly supported Poincar\' e series are of a quite different nature than more classical Poincar\' e series
considered in \cite{Borel1966}, \cite{Borel-SL(2)-book}, \cite{bb} where the Archimedean
group $G_\infty$ must poses representations in discrete series (see the recent works that treat that kinds of series
\cite{kls}, \cite{MuicMathAnn}, \cite{MuicJNT}, \cite{MuicIJNT}, \cite{MuicFCIN}).

Now, we explain the results of the present paper. We let $G$ be a semisimple algebraic group defined
over a number field $k$.  We write $V_f$ (resp.,
$V_\infty$) for the set of
finite  (resp., Archimedean) places.  For $v\in V:=V_\infty\cup V_f$, we write
$k_v$ for the completion of $k$ at $v$. If $v\in V_f$, we let
$\cal O_v$ denote the ring of integers of $k_v$. Let $\bbA$ be the
ring of adeles of $k$. For almost all places of $k$, $G$
is a group scheme over $\mathcal O_v$, and $ G(\cal O_v)$  is a hyperspecial maximal compact
subgroup of $G(k_v)$ (\cite{Tits}, 3.9.1); we say $G$ is unramified over $k_v$.  The group of adelic
points $G(\bbA)=\prod_{v}'
G(k_v) $ is a restricted product over all places of $k$ of the
groups $G(k_v)$. The group $G(\bbA)$ is a locally compact group and
$G(k)$ is embedded diagonally as a discrete subgroup. The group
$G_\infty=\prod_{v\in V_\infty} G(k_v)$ is a semisimple Lie group with finite center but possibly disconnected. We assume that
$G_\infty$ is not compact. We denote by $L^2_{cusp}(G(k)  \setminus G(\bbA))$ a unitary representation of $G(\mathbb A)$ on the
space of all cuspidal $L^2$--functions on $G(k)  \setminus G(\bbA)$ (see Section \ref{prelim} for details). It decomposes
into a direct sum of irreducible unitary representations of $G(\mathbb A)$ called cuspidal automorphic representations.
As opposed to \cite{MuicFCIN} where we deal with underlying Fr\' echet spaces, in this paper we mostly deal with $L^2$ spaces. 

Let $U$ be a unipotent $k$-subgroup of $G$. Let
$\psi: U(k)\setminus U(\mathbb A)\longrightarrow \mathbb C^\times$
be a (unitary) character. We warn the reader that $\psi$ might be trivial. In Section \ref{fc} we define
a $(\psi, U)$--Fourier coefficient of $\varphi\in  L^2(G(k)\setminus G(\bbA))$  by the integral 
\begin{equation}\label{fc-0-int}
{\mathcal F}_{(\psi, U)}(\varphi)(g)=\int_{U(k)\setminus U(\mathbb A)}\varphi(ug)\overline{\psi(u)} du
\end{equation}
which converges almost everywhere  for  $g\in G(\mathbb A)$. We say that $\varphi$ is
$(\psi, U)$--generic if ${\mathcal F}_{(\psi, U)}(\varphi)\neq 0$ (a.e.) for $g\in G$. According to \cite{sh0}, if $G$ is quasi--split over $k$, 
$U$ is the unipotent radical of a Borel
subgroup of $G$ defined over $k$,and $\psi$ is non--degenerate in appropriate sense, then we use the term  $\psi$--generic
instead of $(\psi, U)$--generic. We refer to this settings  as ordinary generic case.

In Section \ref{fc} we adjust the arguments of (\cite{Muic1}, Section 4, Theorem 4.2)
to construct  compactly supported Poincar\' e series with non--zero $(\psi, U)$--Fourier coefficients. We give some details. As an input we have
a finite set of places $S$, containing $V_\infty$, large enough such that $G$, $U$,
and $\psi$ are  unramified for $v\not\in S$, and 
for each $v\in V_f$ we have  $f_v \in C_c^\infty(G(k_v))$ and an open compact subgroup $L_v\subset G(k_v)$ satisfying the following conditions:
\begin{itemize}
\item[(I-a)]  $f_v(1)\neq 0$, for all  $v\in V_f$,
\item[(I-b)] $f_v=1_{G(\calO_v)}$ and  $L_v=G(\calO_v)$ for all $v\not\in S$,
\item[(I-c)] for $v\in  S-V_\infty$, we have
$\int_{U(k_v)} f_v(u_v)\overline{\psi_v(u_v)} du_v\neq 0$,
\item[(I-d)] and, for each $v\in S-V_\infty$, we require that $f_v$ is right--invariant under $L_v$.
\end{itemize}
Then, as an output, we find  $f_\infty \in
C_c^\infty(G_\infty)$ such that  if we let  $f=f_\infty\otimes_{v\in V_f} f_v\in C_c^\infty(G(\mathbb A))$, then the compactly supported
Poincar\' e series
\begin{equation}\label{FC-i}
P(f)(g)=\sum_{\gamma\in G(k)} f(\gamma g), \ \ g\in G(\mathbb A),
\end{equation}
satisfies 
\begin{itemize}
\item[(I-i)]  ${\mathcal F}_{(\psi, U)}(P(f))(1)\neq 0$. In particular, $P(f)$ is a non--zero element of
$L^2(G(k)\setminus G(\mathbb A))^L$, where the open compact subgroup $L$ is defined by $L=\prod_{v\in V_f}L_v$, and $P(f)$ is   $(\psi, U)$--generic.
\item[(I-ii)] $P(f)|_{G_\infty}\neq 0$ and is an element of $L^2(\Gamma_L\setminus G_\infty)$ where $\Gamma_L$
is a congruence subgroup
which corresponds to $L$ from (i) (see (\ref{discreteS-01})).
\item[(I-iii)] $\int_{\Gamma_L\cap U_\infty\setminus U_\infty}P(f)(u_\infty)\overline{\psi_\infty(u_\infty)}du_\infty\neq 0$, where $U_\infty=\prod_{v\in V_\infty} U(k_v)$.
\end{itemize}
The reader may observe that among conditions (I-a)--(I-d), only the conditions (I-c) and (I-d) are delicate.
First, we explain how to assure (I-c) and what are the consequences of (I-i). Later we explain how to deal with
(I-d) and  what are the consequences of (I-ii) and (I-iii).

In Section \ref{lnd}, we fix $v\in V_f$ and consider local $(\psi_v, U(k_v))$--generic representations. Using Bernstein
theory \cite{Be1}, we show how to construct functions $f_v$ satisfying the conditions (I-c) while at the same time we control
the smooth module generated by $f_v$ under right translations. Lemma \ref{lnd-1} contains the result regarding the relation
between non--vanishing of Fourier coefficients and theory of Bernstein classes (it generalizes (\cite{Muic1}, 
Lemma 5.2)). We end Section \ref{lnd} with a result (see Theorem \ref{lnd-4})
regarding the decomposition of algebraic compactly induced representation $c-\Ind_{U(k_v)}^{G(k_v)}(\psi_v)$ (its smooth contragredient is
$\Ind_{U(k_v)}^{G(k_v)}(\overline{\psi}_v)$)  according to Bernstein classes in ordinary generic case (see above).  It  uses global methods of Section \ref{gl}
which in turn rely on above construction of Poncar\' e series in a special case.

In Section \ref{gl}, we prove the main global results. Before describe them we introduce some notation.   In Section \ref{gl}, we define
notion of a $(\psi, U)$--generic $G(\mathbb A)$--irreducible closed subspace of $L^2_{cusp}(G(k)  \setminus G(\bbA))$ as follows. 
First, we define a  closed subrepresentation 
$$
L^2_{cusp,  \text{$(\psi, U)$--degenerate}}(G(k)\setminus G(\bbA))=\left\{\varphi \in
L^2_{cusp}(G(k)  \setminus G(\bbA)); \ \ \text{$\varphi$ is not $(\psi, U)$--generic} \right\}.
$$
Then,  an irreducible closed subrepresentation $\mathfrak U$ of $L^2_{cusp}(G(k)  \setminus G(\bbA))$
is $(\psi, U)$--generic if
$$
\mathfrak U\not\subset L^2_{cusp, \ \text{$(\psi, U)$--degenerate}}(G(k)\setminus G(\bbA)).
$$
The reader might be surprised with this definition but passing to $K$--finite vectors $\mathfrak U_K$
($K$ is a maximal compact subgroup of $G(\mathbb A)$) we obtain  usual definition \cite{sh0}. In particular, if we decompose  $\mathfrak U_K$ 
 into restricted tensor product of local representations $\mathfrak U_K=\pi_\infty \otimes_{v\in V_f} \pi_v$, then all local
representations $\pi_v$ ($v\in V_f$) are $(\psi_v, U(k_v))$--generic in usual sense (see Lemma \ref{gl-0}).  Introducing
the notion of $(\psi, U)$--generic representation in this way, makes  possible  to detect the existence of   $(\psi, U)$--generic representations
contributing to the spectral decomposition of  Poincar\' e series $P(f)$ (defined by (\ref{FC-i})).

 We remark here, and this crucial for considerations of Section \ref{gl},  that
by combining local results of Section \ref{lnd} with (\cite{Muic1}, Proposition 5.3) we may control 
local components in (I-c) not only
to assure that the Poincar\' e series $P(f)$ has a non--zero Fourier coefficient (see (I-i) above)
but also that $P(f)\in L^2_{cusp}(G(k)  \setminus G(\bbA))$. 
Finally, after all of these preparations,  the  main result of the present paper is  the following theorem 
(see Theorem \ref{gl-6}):

\begin{Thm}\label{gl-6-int}
Assume that $G$ is a semisimple  algebraic group defined over a number field $k$. Let $U$ be a unipotent $k$-subgroup.
Let $\psi: U(k)\setminus U(\mathbb A)\longrightarrow \mathbb C^\times$
be a (unitary) character.  Let $S$ be a finite set of places, containing $V_\infty$, large enough such that $G$
 and $\psi$ are  unramified for $v\not\in S$ (in particular, $\psi_v$ is trivial on $U(\mathcal O_v)$).
For each finite place $v\in S$,  let $\mathfrak M_v$ be  a   
$(\psi_v, U(k_v))$--generic Bernstein's class (i.e., there is a $(\psi_v, U(k_v))$--generic irreducible representation 
which belongs to that class; see Definition \ref{lnd-0}) such that the following holds:
 if $P$ is a $k$--parabolic subgroup of $G$ such that
a Levi subgroup of $P(k_v)$ contains a conjugate of a Levi subgroup defining  $\mathfrak M_v$ for all finite $v$ in
$S$,  then $P=G$. Then, there exists an irreducible subspace in
$L^2_{cusp}(G(k)\setminus G(\bbA))$ which is $(\psi, U)$--generic such that its
$K$--finite vectors  $\pi_\infty\otimes_{v\in V_f} \pi_v$ satisfy the following:
\begin{itemize}
\item[(i)] $\pi_v$ is unramified for $v\not\in S$.
\item[(ii)] $\pi_v$ belongs to the class $\mathfrak M_v$ for all finite $v\in S$.
\item[(iii)] $\pi_v$ is $(\psi_v, U(k_v))$--generic for all finite $v$.
\end{itemize}
\end{Thm}

\vskip .2in 
In ordinary generic case, the local results of  Rodier (\cite{ro}, \cite{ro1}) are used to reformulate the requirement  
that the   classes $\mathfrak M_v$  are   $(\psi_v, U(k_v))$--generic in its standard  form (see Lemma \ref{lnd-3}).
In this particular case, the theorem is a vast generalization of similar
results of Henniart, Shahidi, and Vign\' eras (\cite{Hen}, \cite{vig}, \cite{Sh}, Proposition 5.1)  about existence of cuspidal automorphic 
 representations with supercuspidal local components. (See 
 Corollary \ref{gl-7} for details.)  This is because our assumption

\vskip .2in 

\noindent {\it   If $P$ is a $k$--parabolic subgroup of $G$ such that
a Levi subgroup of $P(k_v)$ contains a conjugate of a Levi subgroup defining  $\mathfrak M_v$ for all finite $v$ in
$S$,  then $P=G$. }
\vskip .2in
\noindent{} is satisfied if one of the classes is supercuspidal. In general, none of the classes needs to 
be supercuspidal (see \cite{Muic1} for examples). 

Final remark regarding the theorem is about the case in which $U$ is the unipotent radical of a proper
$k$--parabolic subgroup of $G$, and  $\psi$  is trivial. In this case, the assumptions of the theorem taken together 
do not hold (see the text after Lemma \ref{gl-1} for explanation). 
Therefore, the theorem can not be applied to this case. Of course, 
this is expected since constant terms along proper $k$--parabolic subgroups of  cuspidal automorphic forms vanish 
(they are Fourier coefficients in this particular case).    

In Section \ref{sec-6} we deal with (I-d). For $v\in S-V_\infty$, we construct very specific matrix coefficients $f_v$ of
generic local supercuspidal representations of $G(k_v)$ and open compact subgroups $L_v\subset G(k_v)$  such that (I-c) and (I-d)  hold
(see Proposition \ref{allen-1}). We use the results of (\cite{MoyPr1}, \cite{MoyPr2}).  In Theorem \ref{intr-thm} of Section \ref{newforms} 
we use these results along with the methods of \cite{MoyMuic} to prove the existence of certain
$(\psi_\infty, U_\infty)$--generic cuspidal automorphic representations on $L^2_{cusp}(\Gamma_L\setminus G_\infty)$. We use (I-ii) and (I-iii).

The second named author would like to thank the Hong Kong University of Science and Technology for their hospitality during his 
visit in May of 2014 when the first draft of the paper was written. The second named author would also like to thank the University of Utah 
for their hospitality during his visit in May of 2015 when the final version of the  paper was written.

\section{Preliminaries}\label{prelim}

We let $G$ be a semisimple algebraic group defined
over a number field $k$.  We write $V_f$ (resp.,
$V_\infty$) for the set of
finite  (resp., Archimedean) places.  For $v\in V:=V_\infty\cup V_f$, we write
$k_v$ for the completion of $k$ at $v$. If $v\in V_f$, we let
$\cal O_v$ denote the ring of integers of $k_v$. Let $\bbA$ be the
ring of adeles of $k$. For almost all places of $k$, $G$
is a group scheme over $\mathcal O_v$, and $ G(\cal O_v)$  is a hyperspecial maximal compact
subgroup of $G(k_v)$ (\cite{Tits}, 3.9.1); we say $G$ is unramified over $k_v$.  The group of adelic
points $G(\bbA)=\prod_{v}'
G(k_v) $ is a restricted product over all places of $k$ of the
groups $G(k_v)$: $g=(g_v)_{v\in V} \in G(\bbA)$ if and only if $g_v\in G(\cal O_v)$ for
almost all $v$.  $G(\bbA)$ is a locally compact group and
$G(k)$ is embedded diagonally as a discrete subgroup of $G(\bbA)$.

For a finite  subset  $S\subset V$, we let
$$
G_S=\prod_{v\in S} G(k_v).
$$
In addition, if $S$  contains all Archimedean
places $V_\infty$, we let $G^S=\prod_{v\notin S}' G(k_v)$.
Then
\begin{equation}\label{deeecompSSS}
G(\bbA)=G_S\times G^S.
\end{equation}
We let
$G_\infty=G_{V_\infty}$ and
$G(\bbA_f)=G^{V_\infty}$.

Let $S\subset V$ be a finite set of places containing $V_\infty$ such that $G$ is unramified over
$k_v$. For each $v\in V_f$ we select an open--compact subgroup $L_v$ such that
$L_v=G(\mathcal O_v)$ for all $v\not\in S$. We define an open compact subgroup $L\subset G(\mathbb A_f)$ as follows:
$$
L=\prod_{v\in V_f} L_v.
$$
We consider $G(k)$   embedded diagonally in $G^S$ and define
$$
\Gamma_S=\left(\prod_{v\not\in S}G(\calO_v)\right)\cap G(k).
$$
This can be considered as a  subgroup of $G_S$ using the diagonal embedding of $G(k)$ into the product (\ref{deeecompSSS}) and then the projection to
the first component. Since $G(k)$ is a discrete subgroup of $G(\bbA)$, it follows  that $\Gamma_S$ is a
discrete subgroup of $G_S$. In particular for $S=V_\infty$,  considering $G(k)$   embedded diagonally in $G(\mathbb A_f)$, we define
\begin{equation}\label{discreteS-01}
\Gamma_L=L \cap G(k),
\end{equation}
where $L$ is any open--compact subgroup of $G(\mathbb A_f)$. We obtain a discrete subgroup of $G_\infty$ called a congruence subgroup.

The topological space $G(k)\setminus G(\bbA)$ has a finite volume
$G(\bbA)$--invariant measure:

\begin{equation}\label{invmeas-adel}
 \int_{G(k)\setminus  G(\bbA)} P(f)(g)dg \overset{def}{=}\int_{G(\bbA)}f(g)dg, \ \ f \in C^\infty_c(G(\bbA)),
\end{equation}
where the adelic compactly supported Poincar\' e series  $ P(f)$ is defined as follows:
\begin{equation}\label{adel-CCC-PPP}
P(f)(g)=
\sum_{\gamma\in G(k)}  f(\gamma\cdot g) \in
C_c^\infty(G(k)\setminus G(\bbA)).
\end{equation}
We remark that the space $ C_c^\infty(G(k)\setminus G(\bbA))$ is a subspace of 
$C^\infty(G(\bbA))$ consisting of all functions which are $G(k)$--invariant on the left and
which are compactly supported modulo $G(k)$.

The measure introduced in (\ref{invmeas-adel}) enables us to introduce the Hilbert
space  $ L^2(G(k)\setminus G(\bbA))$, where the inner product is the usual Petersson inner product
$$
\langle \varphi, \ \psi\rangle= \int_{G(k)\setminus G(\mathbb A)}\varphi(g)\overline{\psi(g)}dg.
$$
It is a  unitary representation  of
$G(\bbA)$ under right translations. Next, we define a closed subrepresentation $L^2_{cusp}(G(k)  \setminus G(\bbA))$ consisting of
all cuspidal functions. We recall the definition of  $L^2_{cusp}(G(k)  \setminus G(\bbA))$ and its basic properties.

Since $G(k)\setminus G(\bbA)$ has a finite volume, H\" older inequality implies that 
$ L^2(G(k)\setminus G(\bbA))$ is a subset of $ L^1(G(k)\setminus G(\bbA))$. Every function $\varphi\in  L^1(G(k)\setminus G(\bbA))$ is
 locally integrable on  $G(\bbA)$. This means that for every compact set $C\subset G(\mathbb A)$ we have 
$\int_C |\varphi(g)|dg<\infty$. Next, if $U$ is a $k$--unipotent subgroup of $G$, then $U(k)\setminus U(\mathbb A)$ is compact. Thus, there 
exists a compact neighborhood $D$ of identity of $U(\mathbb A)$ such that $U(\mathbb A)=U(k) D$. Then, for every compact set $C\subset G(\mathbb A)$  we have 

\begin{align*}
\int_C |\varphi(g)| dg&= \int_{U(\mathbb A)\setminus U(\mathbb A)C}\left(\int_{U(k)\setminus U(\mathbb A)} |\varphi(ug)|\sum_{\gamma\in U(k)} 1_C(\gamma ug)du\right)dg\\
&\ge
\int_{U(\mathbb A)\setminus U(\mathbb A)C}\left(\int_{U(k)\setminus U(\mathbb A)} |\varphi(ug)| du\right)dg.
\end{align*}
Letting $C$ vary, this implies 
$$
\int_{U(k)\setminus U(\mathbb A)} \left|\varphi(ug)\right|du <\infty, \ \ \text{(a.e.) for $g\in G(\mathbb A)$.}
$$
If $P$ is a  $k$--parabolic  subgroups  of $G$,
then we denote by $U_P$ the unipotent radical of $P$. For  $\varphi\in  L^1(G(k)\setminus G(\bbA))$, the constant term is a function 
$$
\varphi_P(g)=\int_{U_P(k)\setminus U_P(\mathbb A)} \varphi(ug) du
$$
defined almost everywhere on $G(\mathbb A)$. We say that $\varphi$  is a cuspidal function if  $\varphi_P=0$ almost everywhere on 
$G(\mathbb A)$ for all proper $k$--parabolic subgroups of $G$. Later in the paper we need   compactly supported Poincar\' e series
which are cuspidal functions. Their construction is a rather delicate.
 Using theory of Bernstein classes \cite{Be1} and smooth representation theory of $p$--adic groups 
 we describe fairly general construction of such functions in (\cite{Muic1}, Proposition 5.3).  We use this construction later in the proofs of our
 main results. A different construction of such functions which are spherical has been done by Lindenstrauss and Venkatesh \cite{lindenven}. 
They rely on Satake isomorphism.

We continue with the description of $L^2_{cusp}(G(k)  \setminus G(\bbA))$. 
The space  $L^2_{cusp}(G(k)  \setminus G(\bbA))$ consists of  all cuspidal functions in  $L^2(G(k)  \setminus G(\bbA))$. Obviously, it is 
$G(\mathbb A)$--invariant. It is closed since it is exactly the  subspace of $L^2(G(k)  \setminus G(\bbA))$ orthogonal to all pseudo--Eisenstein 
series
$$
E(\eta, P)(g)=\sum_{U_P(k)\setminus G(k)} \eta(\gamma g), \ \ g\in G(\mathbb A),
$$
where $P$ ranges over all proper $k$--parabolic subgroups of $G$, and 
$\eta\in C_c(U_P(\mathbb A)\setminus G(\mathbb A))$. This follows immediately from the following integration formula:
$$
\langle \varphi, \ E(\eta, P) \rangle= \int_{U_P(\mathbb A)\setminus G(\mathbb A)} \varphi_P(g)\overline{\eta(g)}dg.
$$
We remark that since $U_P(k)\setminus U_P(\mathbb A)$ is compact, we have that $\eta$
is compactly supported modulo $U(k)$. Consequently, we have 
 $E(\eta, P)\in L^2(G(k)  \setminus G(\bbA))$.

We have the following result from the representation theory:

\begin{Thm}\label{thm-cusp-adel} The space $L^2_{cusp}(G(k)
  \setminus G(\bbA))$  can be decomposed into a direct sum of
  irreducible unitary representations of  $G(\bbA)$ each
occurring with a finite multiplicity.
\end{Thm}

\section{Fourier Coefficients and Non--vanishing of Poincar\' e Series}\label{fc}

We begin the section with the following standard definition (see \cite{sh0}, Section 3 for generic case).
Let $U$ be a unipotent $k$-subgroup of $G$. Let $\psi: U(k)\setminus U(\mathbb A)\longrightarrow \mathbb C^\times$
be a (unitary) character. We warn the reader that $\psi$ might be trivial. As with the constant term recalled in Section \ref{prelim},
 for $\varphi\in  L^2(G(k)\setminus G(\bbA))$, the integral
\begin{equation}\label{fc-0}
{\mathcal F}_{(\psi, U)}(\varphi)(g)=\int_{U(k)\setminus U(\mathbb A)}\varphi(ug)\overline{\psi(u)} du
\end{equation}
converges almost everywhere  for  $g\in G(\mathbb A)$. We say that $\varphi$ is
 $(\psi, U)$--generic if ${\mathcal F}_{(\psi, U)}(\varphi)\neq 0$ (a.e.) for $g\in G(\mathbb A)$. 

It follows from (\ref{fc-0}) that
\begin{equation}\label{fc-00}
  {\mathcal F}_{(\psi, U)}(\varphi)(ug)=\psi(u){\mathcal F}_{(\psi, U)}(\varphi)(g), \ \ u\in U(\mathbb A),
  \ \text{(a.e.) for $ g\in G(\mathbb A)$}.
\end{equation}

The space defined by 
$$
L^2_{ \text{$(\psi, U)$--degenerate}}(G(k)\setminus G(\bbA))=\left\{\varphi \in
L^2(G(k)  \setminus G(\bbA)); \ \ \text{$\varphi$ is not $(\psi, U)$--generic} \right\}.
$$
is closed and $G(\mathbb A)$--invariant.  The later is obvious, while the former follows  as in Section \ref{prelim}  
where we discussed
$L^2_{cusp}(G(k)\setminus G(\bbA))$. Indeed, we let
\begin{equation}\label{fc-0011}
E(\eta)(g)=\sum_{U(k)\setminus G(k)} \eta(\gamma g), \ \ g\in G(\mathbb A),
\end{equation}
where $\eta\in C^\infty(G(\mathbb A))$ satisfies the following conditions:
\begin{itemize}
\item $\eta(ug)=\psi(u)\eta(g)$, \ \  $u\in U(\mathbb A), \  g\in G(\mathbb A)$,
\item there exists a compact subset $C\subset G(\mathbb A)$ (depending on $\eta$) such that
  $\supp{(\eta)}\subset U(\mathbb A)\cdot C$.
\end{itemize}
Since $U(k)\setminus U(\mathbb A)$ is compact, we have that $\eta$
is compactly supported modulo $U(k)$. Consequently, we have 
 $E(\eta)\in L^2(G(k)  \setminus G(\bbA))$.  Finally, $\varphi$ is not $(\psi, U)$--generic if and only 
if it is orthogonal to all
 $E(\eta)$.
This follows immediately from the following integration formula:
\begin{equation}\label{fc-0012}
\langle \varphi, E(\eta)\rangle= \int_{U(\mathbb A)\setminus G(\mathbb A)} {\mathcal F}_{(\psi, U)}(\varphi)(g)\overline{\eta(g)}dg
\end{equation}
whose simple proof we leave as an exercise to the reader.

After these preliminary claims, we turn our attention to construction of 
 compactly supported Poincar\' e series  having non--zero $(\psi,  U)$--Fourier coefficients. We need them
in Sections \ref{gl} and \ref{newforms}  for the proof of our main results.

\begin{Lem}\label{fc-1} Let $G$ be a semisimple group defined over $k$. Let $U$ be a unipotent $k$-subgroup.
Let $\psi: U(k)\setminus U(\mathbb A)\longrightarrow \mathbb C^\times$
be a (unitary) character.  Let $S$ be a finite set of places, containing $V_\infty$, large enough such that $G$, $U$,
 and $\psi$ are  unramified for $v\not\in S$ (in particular, $\psi_v$ is trivial on $U(\mathcal O_v)$).
Assume that for each $v\in V_f$ we have  $f_v \in C_c^\infty(G(k_v))$ and an open compact subgroup $L_v$ such that 
\begin{itemize}
\item[(a)] $f_v=1_{G(\calO_v)}$ and  $L_v=G(\calO_v)$ for all $v\not\in S$,
\item[(b)] for $v\in  S-V_\infty$, we have
$\int_{U(k_v)} f_v(u_v)\overline{\psi_v(u_v)} du_v\neq 0$,
\item[(c)] and, for each $v\in S-V_\infty$, we require that $f_v$ is right--invariant under $L_v$.
\end{itemize}
Then, we can find $f_\infty \in
C_c^\infty(G_\infty)$ such that when we let  $f=f_\infty\otimes_{v\in V_f} f_v$  the following holds:
\begin{itemize}
\item[(i)]  ${\mathcal F}_{(\psi, U)}(P(f))(1)\neq 0$. In particular, $P(f)$ is a non--zero element of
$L^2(G(k)\setminus G(\mathbb A))^L$, where the open compact subgroup $L$ is defined by $L=\prod_{v\in V_f}L_v$.
\item[(ii)] $P(f)|_{G_\infty}\neq 0$ and is an element of $L^2(\Gamma_L\setminus G_\infty)$ where $\Gamma_L$
is a congruence subgroup
which corresponds to $L$ from (i) (see (\ref{discreteS-01})).
\item[(iii)] $\int_{\Gamma_L\cap U_\infty\setminus U_\infty}P(f)(u_\infty)\overline{\psi_\infty(u_\infty)}du_\infty\neq 0$.
\end{itemize}
\end{Lem}
\begin{proof}  Since $U(k)\setminus U(\mathbb A)$ is compact, there exists a compact set $C\subset U(\mathbb A)$ such that
$U(\mathbb A)=U(k)C$. We explain how we can choose this set more precisely. First, by
the strong approximation, we have
$$
U(\mathbb A_f)=U(k)\left(L\cap  U(\mathbb A_f)\right).
$$
We consider the decomposition
$$
U(\mathbb A)=U_\infty\times U(\mathbb A_f),
$$
with $U(k)$ diagonally embedded.  Then, we define the continuous map

$$
U_\infty\times \left(L\cap  U(\mathbb A_f)\right)\longrightarrow
U(k)\setminus U(\mathbb A)
$$
given by
$$
(u_\infty, l)\mapsto U(k)(u_\infty, l).
$$
By the strong approximation, this map is surjective and it induces a homeomorphism of topological
spaces
$$
\Gamma_L\cap U_\infty\setminus U_\infty \times \left(L\cap  U(\mathbb A_f)\right)\longrightarrow
U(k)\setminus U(\mathbb A).
$$
This implies that $\Gamma_L\cap U_\infty\setminus U_\infty$ is compact. In particular, we can select a compact set
$C_\infty  \subset U_\infty$ such that
$$
U_\infty= U(k)C_\infty.
$$
Hence, this implies that we can select a compact set
\begin{equation}\label{fc-2}
C=C_\infty \times \left(L\cap  U(\mathbb A_f)\right).
\end{equation}
in order to obtain $U(\mathbb A)=U(k)C$.

Since $G(k)$ is discrete in $G(\mathbb A)$ and the set (see (\ref{fc-2}))
\begin{equation}\label{fc-3}
D\overset{def}{=}C^{-1}_\infty \times \prod_{v\in S-V_\infty}  \supp{(f_v)} \times \prod_{v\not\in S}
G(\mathcal O_v)
\end{equation}
compact, we have  that the set $G(k)\cap D$ is finite. We claim that
\begin{equation}\label{fc-4}
G(k)\cap D \subset U(k).
\end{equation}
Indeed, considering the projection to the first factor in
(\ref{fc-3}), we find that
$$
G(k)\cap D \subset C^{-1}_\infty
$$ 
when we consider $G(k)$ as a subgroup of $G_\infty$. But
$C^{-1}_\infty \subset U_\infty $. So that
$$ G(k)\cap D \subset C^{-1}_\infty\cap G(k)\subset U_\infty\cap G(k)=U(k).
$$
This proves (\ref{fc-4}). Next,  we can find an open set $V'_\infty$ in $G_\infty$
containing $C^{-1}_\infty$ such that
$$
G(k)\cap \left(V'_\infty \times \prod_{v\in S-V_\infty}  \supp{(f_v)} \times \prod_{v\not\in S}
G(\mathcal O_K)\right)=G(k)\cap D.
$$
We select an open neighborhood $V_\infty$ of identity in $G_\infty$ such that
$V_\infty\cdot C^{-1}_\infty \subset  V'_\infty$.
In particular,
\begin{equation}\label{fc-5}
G(k)\cap \left(V_\infty\cdot C^{-1}_\infty  \times \prod_{v\in S-V_\infty}
\supp{(f_v)} \times \prod_{v\not\in S}
G(\mathcal O_K)\right)=G(k)\cap D \subset U(k).
\end{equation}
Next, we select $f_\infty \in C_c^\infty(G_\infty)$ such that
$\supp{(f_\infty)}\subset V_\infty$,  and
\begin{equation}\label{fc-6}
\int_{U_\infty} f_\infty(u_\infty)\overline{\psi_\infty(u_\infty)} du_\infty\neq 0.
\end{equation}
This can be achieved by requiring that support of $f_\infty$ is small enough so that it is contained in the image of the restriction of $\exp$ 
to a small neighborhood of $0\in   \mathfrak g_\infty$ where that restriction is a diffeomorphism onto its image. Then, we can transfer statement (\ref{fc-6}) 
to the Lie algebra by writting the the Haar measure on $U_\infty$ in local  coordinates (it as differential form of top degree which never vanish). The 
obtained claim is easy to verify directly.

Now, we are ready to prove (i). We compute
$$
{\mathcal F}_{(\psi, U)}(P(f))(1)= \int_{U(k)\setminus U(\mathbb A)} \sum_{\gamma\in G(k)}
f(\gamma\cdot u)\overline{\psi(u)} du.
$$
We reduce above expression using the following observation:
\begin{equation}\label{fc-10000}
f(\gamma\cdot u)\neq 0, \ \  \text{for some $u\in U(\mathbb A)$ and $\gamma\in  G(k)$, implies } \ \ 
\gamma \in U(k).
\end{equation}
Let us prove (\ref{fc-10000}).  Using $U(\mathbb A)=U(k)C$,
we can write $u=\delta c$ where $\delta\in U(k)$ and $c\in C$. Since $f(\gamma\cdot u)\neq 0$, we
obtain
$$
\gamma \delta \in G(k)\cap \supp{(f)}\cdot C^{-1}.
$$
The key observation is that the assumptions (a) and (c) from the statement of the lemma as well as (\ref{fc-2})
imply
$$
\supp{(f)}\cdot C^{-1}=\supp{(f_\infty)}\cdot C^{-1}_\infty\times  \prod_{v\in S-V_\infty} \supp{(f_v)}
\times \prod_{v\not\in S} G(\mathcal O_v)
$$
Using this and  $\supp{(f_\infty)}\subset V_\infty$, (\ref{fc-5}) implies that
$$
G(k)\cap \supp{(f)}\cdot C^{-1}\subset U(k).
$$
Which shows  $\gamma\delta \in U(k)$. Hence $\gamma\in U(k)$. This proves (\ref{fc-10000}).

Using (\ref{fc-10000}), (b) from the statement of the lemma, and (\ref{fc-6}),
above integral becomes
\begin{align*}
{\mathcal F}_{(\psi, U)}(P(f))(1)&= \int_{U(k)\setminus U(\mathbb A)} \sum_{\gamma\in
  U(k)} f(\gamma\cdot u)\overline{\psi(u)} du \\ &=\int_{U(\mathbb A)}
f(u)\overline{\psi(u)} du\\ &=\left(\int_{U(k_\infty)}
f_\infty(u_\infty)\overline{\psi_\infty(u_\infty)} du_\infty\right) \prod_{v\in
  S-V_\infty} \int_{U(k_v)} f_v(u_v)\overline{\psi_v(u_v)} du_v \neq 0.
\end{align*}
This implies (i) in view of the assumption (c).

To prove (ii) and (iii), we recall that in (\cite{Muic1}, Proposition
3.2) we prove that $P(f)|_{G_\infty}$ is a compactly supported
Poincar\' e series on $G_\infty$ for $\Gamma_L$. This shows that it belongs
to $L^2(\Gamma_L\setminus G_\infty)$.  In order to complete the proofs
of (ii) and (iii), we observe that (b) implies that each $\psi_v$ is
invariant under $L_v\cap U(k_v)$. This means that $P(f)|_{U(\mathbb A_f)}\left(\otimes_{v\in V_f}
\overline{\psi_v}\right)$ is right invariant under $L\cap U(\mathbb A_f)$. This enables us to apply
(\cite{Muic1}, Lemma  3.3):
$$
{\mathcal F}_{(\psi, U)}(P(f))(1)=vol_{U(\mathbb A_f)}\left(L\cap U(\mathbb A_f)\right)\cdot
\int_{\Gamma_{L}\cap U_\infty  \setminus
U_{\infty}}P(f)(u_\infty)\overline{\psi_\infty(u_\infty)}du_\infty.
$$
In view of (i), this proves (ii) and (iii).
\end{proof}

\section{Local Generic Representations} \label{lnd}
In this section we discuss local generic
representation. We drop index $v$, and let $k$ be a non--Archimedean local field.
We assume that $G$ is a semisimple group defined over $k$. We write $G$ for $G(k)$ in
order to simplify notation.
Similarly, we do for subgroups of $G$. The goal of this section is to explain how to construct functions satisfying
Lemma  \ref{fc-1} (b) using theory of Bernstein \cite{Be1}. The reader may also want to consult
(\cite{Muic1}, Section 5). In the present section we refine  some of the results proved there
 for our particular application.

We introduce some notation following standard references \cite{BZ} and \cite{BZ1}.
We consider the category of all smooth complex representations of $G$. For a smooth representation $\pi$,
we denote $\wit{\pi}$ the smooth dual of $\pi$. We call it a contragredient representation.

Let $P$ be a parabolic subgroup
of $G$ given by a Levi decomposition $P=MU_P$, where
$M$ is a Levi factor and $U_P$ is the unipotent radical of $P$.
If $\sigma$ is a smooth representation of $M$ extended trivially
across $U_P$ to a representation of $P$, then we denote  the normalized
induction by $\Ind_{P}^{G}(\sigma)$. If $\pi$ is a smooth
 representation of  $G$, then we denote by $Jacq^{P}_{G}(\pi)$
a normalized Jacquet module of $\pi$ with respect to $P$. When restricted to $U_P$,
 $Jacq^{P}_{G}(\pi)$ is a direct sum of (possibly infinitely many)
copies of a trivial representation.  Therefore, when $M$ is fixed,
 we write $Jacq^{M}_{G}(\pi)=Jacq^{P}_{G}(\pi)$. Let $|\ |$ be an absolute value on $k$.
Let $M^0$ be the subgroup of $M$ given as the
 intersection of the kernels of all characters $m\mapsto
 |\chi(m)|$, where $\chi$ ranges over the group of all
 $k$--rational algebraic characters $M\rightarrow
 k^\times$. We say that a character $\chi: M\rightarrow
 \mathbb C^\times$ is unramified if it is trivial on $M^0$. We say that
 an irreducible representation $\rho$ of $M$ is supercuspidal if
$Jacq^{Q}_{M}(\rho)=0$ for all proper parabolic subgroups $Q$
 of $M$.

We recall Bernstein's decomposition of the
category of smooth complex
representations of $G$ \cite{Be1}. On the set of
pairs $(M, \rho)$, where $M$ is a Levi subgroup of
$G$  and $\rho$ is a smooth irreducible supercuspidal representation of $M$,
we introduce the relation of equivalence as follows:
$(M, \rho)$  and $(M', \rho')$  are equivalent if we can find $g\in
G$ and  an unramified character  $\chi$  of $M'$ such that $M'=gMg^{-1}$ and $\rho'\simeq
\chi\rho^{g}$ i.e.,
$$
\rho^{g}(m')=\chi(m')\rho(g^{-1}m'g), \ \ m'\in M'.
$$ We write $\left[M, \rho\right]$ for the Bernstein's equivalence
class associated to a pair $(M, \rho)$. We say that a class $\left[M,
  \rho\right]$ is supercuspidal if $M=G$. The contragredient Bernstein's class $\wit{\mathfrak M}$ of
the class $\mathfrak M=\left[M, \rho\right]$ is the class $\left[M, \wit{\rho}\right]$

Let $V$ be a smooth complex representations of $G$. Let
$$
V(\left[M, \rho\right])
$$
be the largest smooth submodule of $V$ such that every irreducible
subquotient of $V$ is a subquotient of  $\Ind_{P}^{G}(\chi \rho)$,
for some unramified character $\chi$ of $M$. Here $P$ is an arbitrary
parabolic subgroup of $G$ containing $M$ as a Levi subgroup. The
fundamental result of Bernstein is the following decomposition:
$$
V=\oplus_{\mathfrak M} V(\mathfrak M),
$$
where $\mathfrak M$ ranges over all Bernstein equivalence classes.

 We say that a smooth
representation $\pi$ of $G$ belongs to the class $\left[M,
  \rho\right]$ if the following holds:
$$
\pi(\left[M, \rho\right])=\pi.
$$
It is obvious that any non--zero subquotient of $\pi$ belongs to  the same class.
It is well--known each irreducible representation $\pi$ belongs to a unique
Bernstein's class.

Now, we apply this theory to study generic representations. We consider the following very general set-up.
Later in the section we give examples.
 Let $U$ be any unipotent $k$--subgroup of $G$ and let
$\chi: U\longrightarrow \mathbb C^\times$ be a character. Since $U$ is a union of open compact subgroups, $\chi$ is
unitary. For the same reason, $U$ is unimodular. We form the two types of induced representations (see
(\cite{BZ, BZ1}):

\vskip .2in

\noindent{1)} $\Ind_{U}^{G}(\chi)$ on the space of all functions $f:
G\longrightarrow \mathbb C$ satisfying $f(ug)=\chi(u)f(g)$, for all
$g\in G$, $u\in U$, and there exists an open--compact subgroup $L$
such that $f(gl)=f(g)$, for all $g\in G$, $l\in L$.

\vskip .2in

\noindent{2)}  $\cInd_{U}^{G}(\chi)$ on the space of all  functions $f\in \Ind_{U}^{G}(\chi)$
which are compactly supported modulo $U$.

\vskip .2in
The contragredient  of the representation  $\cInd_{U}^{G}(\overline{\chi})$ is
$\Ind_{U}^{G}(\chi)$. The canonical pairing

$$ \cInd_{U}^{G}(\overline{\chi})\times
\Ind_{U}^{G}(\chi)\longrightarrow \mathbb C
$$
is
 given by
$$
\langle f, \ F\rangle=\int_{U\setminus G}  f(g)F(g)dg.
$$
\vskip .2in
Let $\pi$ be a smooth representation of $G$. Let $V$ be the space on which $\pi$ acts. Let $V(U, \chi)$ to be the span
of all vectors $\pi(u)v-\chi(u)v$, $v\in V$. Put $r_{U, \chi}(V)=V/V(U, \chi)$. It is the largest quotient of $V$ on which
$U$ acts as $\chi$. The assignment $V\mapsto r_{U, \chi}(V)$ can be considered as a functor from the category of smooth $G$--representations to the
category of smooth $U$--representations. Since $U$ is the union of open compact subgroups, the functor is exact
(\cite{BZ}, Proposition 2.3.5).
The following definition is standard. Let $\pi$ be a smooth representation of $G$. We say that $\pi$  is
 $(\chi, U)$--generic if
$$
\Hom_{G}\left(\pi, \ \Ind_{U}^{G}(\chi)\right)\neq 0.
$$
By Frobenius reciprocity, this is equivalent to
$$
r_{U, \chi}(\pi)\neq 0.
$$
\

\begin{Def}\label{lnd-0}
Let $\mathfrak M$ be a Bernstein's class. We say that $\mathfrak M$
is $(\chi, U)$--generic if  there exists an irreducible representation in this class which is
 $(\chi, U)$--generic. 
\end{Def}

In the settings of Definition \ref{lnd-0}, we have the following simple result:
\begin{Lem}\label{lnd-0a}  Let $\mathfrak M$ be a Bernstein's class. Then we have the following:
  \begin{enumerate} 
  \item[(i)] If  $\cInd_{U}^{G}(\chi)(\mathfrak M)\neq 0$, then $\Ind_{U}^{G}(\chi)(\mathfrak M)\neq 0$.
\item[(ii)] If the class $\mathfrak M$ is $(\chi, U)$--generic, then $\Ind_{U}^{G}(\chi)(\mathfrak M)\neq 0$.
\item[(iii)] The class $\mathfrak M$ is $(\chi, U)$--generic
if and only if  $\Ind_{U}^{G}(\chi)(\mathfrak M)$ has an irreducible subrepresentation, or,
equivalently, $\cInd_{U}^{G}(\overline{\chi})(\wit{\mathfrak M})$ has an irreducible quotient.

\item[(iv)] If supercuspidal representation $\rho$ is $(\chi, U)$--generic, then 
  $\cInd_{U}^{G}(\chi)([G, \rho])\neq 0$, and  $\cInd_{U}^{G}(\overline{\chi})([G, \wit{\rho}])\neq 0$.
\item[(v)] Conversely, if $\Ind_{U}^{G}(\chi)([G, \rho])\neq 0$, where $\rho$ is a supercuspidal representation,  
then $\rho$ is $(\chi, U)$--generic.
\end{enumerate}
\end{Lem}
\begin{proof} The claims (i) and (ii), and the first claim in (iii) are obvious.
For the second claim in (iii), we note that if $\pi$ is an admissible representation, then
$$
\pi\simeq \wit{\wit{\pi}}.
$$
So, since
$$
\left(\cInd_{U}^{G}(\overline{\chi})\right)^\sim\simeq \Ind_{U}^{G}(\chi),
$$
we obtain

$$
\Hom_{G}\left(\pi, \ \Ind_{U}^{G}(\chi)\right)
 \simeq \Hom_{G}\left(\wit{\wit{\pi}}, \ \left(\cInd_{U}^{G}(\overline{\chi})\right)^\sim\right)\simeq
\Hom_{G}\left(\cInd_{U}^{G}(\overline{\chi}), \ \ \wit{\pi}\right).
$$
From this observation,  the second claim easily follows.
Let us prove (iv). By our assumption
$$
\Hom_G\left(\rho, \ \Ind_{U}^{G}(\chi)\right)\neq 0.
$$
By computation in (iii), this implies that
\begin{equation}\label{lnd-1000}
\Hom_{G}\left(\cInd_{U}^{G}(\overline{\chi}),
\ \ \wit{\rho}\right)\neq 0.
\end{equation}
Since the group $G$ has finite center (being semisimple), $\rho$
is a projective object in the category of all smooth
representations of $G$. Thus, (\ref{lnd-1000}) implies
\begin{equation}\label{lnd-1001}
\Hom_{G}\left(\wit{\rho}, \ \
 \cInd_{U}^{G}(\overline{\chi})\right)\neq 0.
\end{equation}
This implies that
$$
\cInd_{U}^{G}(\overline{\chi})([G, \wit{\rho}])\neq 0.
$$
Next, obviously (\ref{lnd-1001}) implies
\begin{equation}\label{lnd-1002}
\Hom_{G}\left(\wit{\rho}, \ \
 \Ind_{U}^{G}(\overline{\chi})\right)\neq 0.
\end{equation}
So again, by the proof of (iii), we obtain
from (\ref{lnd-1002}) the following
$$
\Hom_{G}\left(\cInd_{U}^{G}(\chi),
\ \ \rho\right)\simeq
\Hom_{G}\left(\cInd_{U}^{G}(\chi),
\ \ \wit{\wit{\rho}}\right)\neq 0.
$$
Thus, finally applying the projectivity argument one
more time, we obtain
$$
\Hom_{G}\left(\rho, \ \
 \Ind_{U}^{G}(\chi)\right)\neq 0.
$$
This completes the proof of (iv). The claim (v) follows from the fact that $\rho$
is projective object in the category of smooth representations of $G$.
 \end{proof}

In the case of usual $\chi$--generic representations (see the text after the proof of Lemma \ref{lnd-1} below),
part (iv) has been proved earlier by Casselman and Shalika
(\cite{cs}, Corollary 6.5).

Now, we use Bernstein's theory to show existence of certain types of functions with non--vanishing Fourier coefficients.
This will be crucial in Section \ref{gl} for global applications.

Let $f\in C_c^\infty(G)$. Then, we define a Fourier coefficient of $f$ along $U$ with respect to $\chi$ as follows:
$$
{\mathcal F}_{(\chi, U)}(f)(g)=\int_{U}f(ug)\overline{\chi(u)}du, \ \ g\in G.
$$
Clearly
$$
{\mathcal F}_{(\chi, U)}(f)\in \cInd_{U}^{G}(\chi).
$$

\begin{Lem}\label{lnd-1}
Let $\mathfrak N$   be a  Bernstein's class which satisfies $\cInd_{U}^{G}(\chi)(\mathfrak N)\neq 0$.
(By Bernstein theory there exists at least one such class.)
Then, considering $C_c^\infty(G)$ as a  smooth module under right translations, there exists
$f\in C_c^\infty(G)(\mathfrak N)$ such that the Fourier coefficient ${\mathcal F}_{(\chi, U)}(f)$
is not identically equal to zero.
\end{Lem}
\begin{proof}  We observe the following simple fact. If $V$ and $W$ are smooth representations such that $W$ is a
quotient of $V$. Then, for any Bernstein's class $\mathfrak N$, $W(\mathfrak N)$ is a quotient of $V(\mathfrak N)$.
This follows immediately from the facts that $V$ and $W$ can be decomposed into a direct sum of modules
$V(\mathfrak N)$ and $W(\mathfrak N)$, and 
$V(\mathfrak N)$ is mapped into $W(\mathfrak N)$.

Next, it is the standard fact that the map
$C_c^\infty(G)\longrightarrow
\cInd_{U}^{G}(\chi)$ given by $f\longmapsto {\mathcal F}_{(\chi, U)}(f)$ is surjective intertwining operator of
right regular representations. This clearly implies the surjectivite map
$$
C_c^\infty(G)(\mathfrak N)\longrightarrow
\cInd_{U}^{G}(\chi)(\mathfrak N)
$$
for any Bernstein's class $\mathfrak N$. The lemma follows.
\end{proof}

Now, we list some examples for above theory. First of all, there are various trivial cases such as the
 case $U=\{1\}$ and $\chi$ is trivial, or the case $U=U_P$ and $\chi$ is trivial, for some proper
parabolic subgroup of $G$. The reader may want to compute generic representations in both cases
as an easy exercise.

For global applications (see \cite{Sh}), one instance
 of Lemma \ref{lnd-1} of the greatest importance. Assume that $G$ is quasi-split over $k$. Let $B=TU_B$
be a Borel
subgroups defined over $k$ given by its Levi decomposition, $T$ is a torus and $U_B$ is unipotent radical both defined over $k$. We let $U=U_B$ and assume that
$\chi$ is generic in the sense that $\chi$ is not trivial when restricted to any root subgroup $U_\alpha$, where $\alpha$ is a simple root corresponding to the choice of $B$. It is a 
fundamental result of Rodier \cite{ro, ro1} that $\dim r_{U, \chi}{(\pi)}\le 1$. Moreover, if $P=MU_P$ is a standard parabolic
subgroup of $G$ (i.e., $B\subset P$, $T\subset M$ a standard choice for Levi subgroup; the details can be found in (\cite{cs}, page 208)) and $\sigma$ is a an admissible  
representation of $M$, then we have an isomorphism of vector spaces \cite{cs, ro, ro1}
$$
 r_{U, \chi}{\left(\Ind_{P}^G(\sigma)\right)}\simeq  r_{U\cap M, \chi'}{(\sigma)},
 $$
where $\chi'$ is again a generic character defined by
\begin{equation}\label{lnd-2}
\chi'(u)=\chi(w^{-1}uw), \ \ u\in U\cap M.
\end{equation}
The element $w$ is any element of  $N_G(A)$, where $A$ is a split component in the center of $M$, which
 satisfies that the quotient $P\setminus PwB$ is unique open double coset in
$P\setminus G$. As it is more usual, in this case we speak of $\chi$--generic representations and $\chi$--generic Bernstein classes.
In this case, above discussion implies the following standard lemma which proof we leave to a reader as an exercise.

\begin{Lem}\label{lnd-3}
 Assume that $G$ is quasi-split over $k$. The class $\mathfrak M$ is $\chi$--generic if and only
 if for a representative $(M, \rho)$ of $\mathfrak M$ which is taken among the set of standard Levi subgroups
 we have that $\rho$ is $\chi'$--generic.
\end{Lem}

\vskip .2in
We end this section by the following local result which we prove using global methods from
the next section.

\begin{Thm}\label{lnd-4}
 Assume that $G$ is quasi-split over $k$. Let $\chi$ be a generic character of $U=U_B$.  Let $\mathfrak M$ be  any
 Bernstein's class such that $\cInd_{U}^{G}(\chi)(\mathfrak M)\neq 0$.
 Then, the class $\mathfrak M$ is $\chi$--generic.
\end{Thm}
\begin{proof} Let us make some preliminary reductions to the proof. Let us fix a generic character $\chi_0$ of $U=U_B$.
  Let $\overline{k}$ be the algebraic closure of $k$. Then, as indicated in (\cite{Sh}, Section 3), for each generic
  character $\chi$ of $U$ there exists
an element $a\in A(\overline{k})$, where $A$ is a maximal split $k$--torus in $T$ such that the following holds:
\begin{itemize}

\item  the map $ g\mapsto a^{-1}ga$ is  a continuous automorphism of $G=G(k)$,

\item  $a^{-1}Ua=U$,

\item $\chi(u)=\chi_0(a^{-1}ua)$, for all $u\in U$,

\item it fixes the set of standard parabolic subgroups of $G$  and their standard Levi subgroups
  (with respect to the choice of $B$ and $A$)

\item it permutes the set of supercuspidal representations and the set of unramified characters of each standard
  Levi subgroup $M$:
$\rho^a(m)=\rho(a^{-1}ma)$, and $\chi^a(m)=\chi(a^{-1}ma)$, $m\in M$

\item the map $\pi\longmapsto \pi^a$ permutes irreducible representations

\item if $\pi$ is $\chi_0$--generic, then  $\pi^a$ is $\chi$--generic

 \item if $\pi$ is a subquotient of $\Ind_P^G(\chi\rho)$, then $\pi^a$  is a subquotient of $\Ind_P^G(\chi^a\rho^a)$.
\end{itemize}

These facts show that it is enough to establish the theorem for some convenient character $\chi$. We  complete the proof
using Corollary \ref{gl-8}, Lemma \ref{lnd-5}, and the fact that when $G$ is split then there exist generic supercuspidal
representations of $G$ (see Proposition \ref{allen-1} for the case of simple groups).
\end{proof}

\vskip .2in

\begin{Lem}\label{lnd-5} Let $H$ be a reductive group defined  over a number field $K$. Then there exists infinitely many places $v$ of $K$ such that
$H$ is split over $K_v$.
\end{Lem}
\begin{proof} There exists a finite Galois extension $K\subset L$ such that $H$ splits over $L$ i.e., $H$ has a maximal torus defined over $L$ and split
 over $L$. On the other hand, by Chebotarev density theorem, there exists a set of finite primes $v$ of $K$ of positive density which are split in the sense of algebraic
 number theory with respect to extension $K\subset L$. For such $v$ and a finite place $w|v$ of $L$, we have $K_v=L_w$.
 Since $H$ is obviously split over $L_w$ (being split over $L$),  $H$ is
 split over $K_v$.
 \end{proof}

\section{Main Global Theorems}\label{gl}

In this section we return to the global settings of Section \ref{fc}. Let $K=K_\infty\times \prod_{v\in V_f} K_v$ be a maximal compact
subgroup of $G(\bbA)$, where $K_v=G(\calO_v)$ for almost all $v$. By Theorem \ref{thm-cusp-adel}, $L^2_{cusp}(G(k)  \setminus G(\bbA))$
can be decomposed into a Hilbert direct sum of irreducible unitary representations of G(A) each occurring with a finite multiplicity.
Then, the same is true for any closed subrepresentation of $L^2_{cusp}(G(k)  \setminus G(\bbA))$.

Let $\mathfrak U$ be an irreducible subrepresentation of $L^2_{cusp}(G(k)  \setminus G(\bbA))$. On the space $\mathfrak U_K$ of $K$--finite vectors we have
an irreducible representation $\pi$ of $(\mathfrak g_\infty, K_\infty)\times G(\mathbb A_f)$, where  $\mathfrak g_\infty$ is a real Lie algebra of $G_\infty$.
In fact, $\pi$ is an irreducible subspace of the space of all cuspidal automorphic forms $\cal A_{cusp}(G(k)\setminus G(\mathbb A))$ and it is dense in $\mathfrak U$
(see \cite{BJ}). The representation $\pi$ is a restricted tensor product of local representations: $\pi\simeq \pi_\infty \otimes_{v\in V_f} \pi_v$, where for almost all $v\in V_f$ 
the representation $\pi_v$ is unramified.

Let $U$ be a unipotent $k$-subgroup of $G$. Let $\psi: U(k)\setminus U(\mathbb A)\longrightarrow \mathbb C^\times$
be a (unitary) character.  We define a  closed subrepresentation (see Section \ref{fc})
$$
L^2_{cusp, \ \text{$(\psi, U)$--degenerate}}(G(k)\setminus G(\bbA))=
L^2_{cusp}(G(k)  \setminus G(\bbA))\cap L^2_{\text{$(\psi, U)$--degenerate}}(G(k)\setminus G(\bbA))
$$
Let $\mathfrak U$ be an irreducible closed subrepresentation of $L^2_{cusp}(G(k)  \setminus G(\bbA))$ such that
$$
\mathfrak U\not\subset L^2_{cusp, \ \text{$(\psi, U)$--degenerate}}(G(k)\setminus G(\bbA)).
$$
Then we say that $\mathfrak U$ is $(\psi, U)$--generic.

We have the following standard result:

\begin{Lem}\label{gl-0}
Let $\mathfrak U$ be an irreducible subrepresentation of $L^2_{cusp}(G(k)  \setminus G(\bbA))$
which is $(\psi, U)$--generic. Then, for every $v\in V_f$, the representation $\pi_v$ is  $(\psi_v, U(k_v))$--generic.
\end{Lem}
\begin{proof} Let $\Lambda: \ \mathfrak U_K\longrightarrow \mathbb C$ be a linear functional defined by
$$
\varphi\longmapsto {\mathcal F}_{(\psi, U)}(\varphi)(1)=\int_{U(k)\setminus U(\mathbb A)}\varphi(u)\overline{\psi(u)} du.
$$
We show that $\Lambda$ is non--zero. Assuming this for a moment, we complete the proof of the lemma. Let us fix a finite place $v$. 
Then, for $u\in U(k_v)$ and  $\varphi\in \mathfrak U_K$, we have the following:
$$
\Lambda\left(\pi(u_v)\varphi\right)=  \int_{U(k)\setminus U(\mathbb A)}\pi(u_v)\varphi(u)\overline{\psi(u)} du=
\int_{U(k)\setminus U(\mathbb A)}\varphi(uu_v)\overline{\psi(u)} du=\psi_v(u_v)\Lambda\left(\varphi\right).
$$
This means that $\mathfrak U_K$ is  $(\psi_v, U(k_v))$--generic considered as a smooth $G(k_v)$--representation.
But this representation is a direct sum of possibly infinitely many copies of $\pi_v$. 
This means that $\pi_v$ is  $(\psi_v, U(k_v))$--generic.

It remains to show that $\Lambda\neq 0$. If not, we have 
$${\mathcal F}_{(\psi, U)}(\varphi)(1)=0
$$
for all $\varphi\in \mathfrak U_K$. Since $\mathfrak U_K$ is $(\mathfrak g_\infty, K_\infty)\times G(\mathbb A_f)$--invariant,
writing
$$
G(\mathbb A)=G_\infty\times G(\mathbb A_f),
$$
we conclude that 
$$
{\mathcal F}_{(\psi, U)}(\varphi)(k_\infty \exp{(X)}, g_f)=\sum_{n=0}^\infty \frac{1}{n!}X^n.{\mathcal F}_{(\psi, U)}(\varphi)(k_\infty, g_f)=0,
$$ 
for any $g\in G(\mathbb A_f)$, $k_\infty \in K_\infty$, and for $X$ in a neighborhood of $0$ (depending on $k_\infty$)
in $\mathfrak g_\infty$. This means that there exists an open set $V\subset G_\infty$ which meets all connected components (in usual metric topology) of 
$ G_\infty$ such that 
$$
{\mathcal F}_{(\psi, U)}(\varphi)=0 \ \ \text{on $V\times G(\mathbb A_f)$}.
$$
This implies that 
$$
{\mathcal F}_{(\psi, U)}(\varphi)=0 \ \ \text{on $G(\mathbb A)$}
$$
since ${\mathcal F}_{(\psi, U)}(\varphi)$ is 
real--analytic in the first variable being an integral over a compact set of $\varphi$ which is obviously 
real analytic function in the first variable.

Thus, we conclude that ${\mathcal F}_{(\psi, U)}=0$ on the dense subset $\mathfrak U_K$ of $\mathfrak U$. Let now $\varphi\in \mathfrak U$. Then, 
using the discussion at beginning  of Section \ref{fc} (see 
(\ref{fc-0012})), we conclude that
$\langle \varphi, \eta\rangle=0$ 
for all $\eta$ described there. From this, applying again (\ref{fc-0012}), we conclude that 
 ${\mathcal F}_{(\psi, U)}(\varphi)=0$. Since $\varphi\in \mathfrak U$ is arbitrary, we conclude that $\mathfrak U$ is 
not $(\psi, U)$--generic.
\end{proof}

Now, we state and prove the main technical result of the
present section.

\vskip .2in

\begin{Lem}\label{gl-1}
  Assume that $G$ is a semisimple algebraic group defined over a number field $k$. Let $U$ be an unipotent $k$-subgroup of $G$.
  Let $\psi: U(k)\setminus U(\mathbb A)\longrightarrow \mathbb C^\times$
be a (unitary) character.  Let $S$ be a finite set of places, containing $V_\infty$, large enough such that $G$
 and $\psi$ are  unramified for $v\not\in S$ (in particular, $\psi_v$ is trivial on $U(\mathcal O_v)$).
 For each finite place $v\in S$,  let $\mathfrak M_v$ be  a  Bernstein's class such that
 $\cInd_{U(k_v)}^{G(k_v)}(\psi_v)(\mathfrak M_v)\neq 0$. Assume  the following property:
if $P$ is a $k$--parabolic subgroup of $G$ such that
a Levi subgroup of $P(k_v)$ contains a conjugate of a Levi subgroup defining  $\mathfrak M_v$ for all finite $v$ in
$S$,  then $P=G$. 
Then, there exists an irreducible subspace in
$L^2_{cusp}(G(k)\setminus G(\bbA))$ which is $(\psi, U)$--generic such that its
$K$--finite vectors  $\pi_\infty\otimes_{v\in V_f} \pi_v$ satisfy the following:
\begin{itemize}
\item[(i)] $\pi_v$ is unramified for $v\not\in S$.
\item[(ii)] $\pi_v$ belongs to the class $\mathfrak M_v$ for all finite $v\in S$.
\item[(iii)] $\pi_v$ is $(\psi_v, U(k_v))$--generic for all finite $v$.
\end{itemize}
In particular, for each finite $v\in S$, the class $\mathfrak M_v$ is $(\psi_v, U(k_v))$--generic.
 \end{Lem}

 \vskip .2in
 Before we start the proof we make some preliminary remarks. If $U=\{1\}$ and $\chi=1$, then Lemma \ref{gl-1} is
 just
 (\cite{Muic1}, Theorem 1.1). On the other hand, assuming that $\chi$ is trivial and
 $U$ is a  unipotent radical of a proper $k$--parabolic subgroup
 $Q$ of $G$, 
 our assumptions on $\mathfrak M_v$ (for finite $v\in S$)
 means that there exists a non--zero function $f_v\in C_c^\infty(G(k_v))(\mathfrak M_v)$ such that
 $$
 \int_{U(k_v)}f_v(u_vg_v) du_v\neq 0
 $$
 for some $g_v$.  Then, (\cite{Muic1}, Lemma 5.1) implies that a conjugate of a Levi subgroup defining
 $\mathfrak M_v$ is contained in a
 Levi subgroup of $Q(k_v)$. Since this holds for all $v\in S$, we would get $Q=G$ which is not possible.
 So, in this case, as it should be, the theorem does not give anything.

 \vskip .2in

\begin{proof}[Proof of Lemma \ref{gl-1}] As in Lemma \ref{fc-1}, we let $f_v=1_{G(\calO_v)}$ for all $v\not\in S$.
For finite $v\in S$, applying  Lemma \ref{lnd-1}, we select $f\in C_c^\infty(G(k_v))(\mathfrak M_v)$ such that
$$
\int_{U(k_v)} f_v(u_v)\overline{\psi_v(u_v)} du_v\neq 0.
$$
We  select open compact subgroups $L_v$ ($v\in V_f$)
as required  in Lemma \ref{fc-1}. Then, by Lemma \ref{fc-1}, there exists $f_\infty \in
C_c^\infty(G_\infty)$ such that letting  $f=f_\infty\otimes_{v\in V_f} f_v$  we have

$$
{\mathcal F}_{(\psi, U)}(P(f))\neq 0.
$$
Thus, $P(f)$ is a non--zero element of $L^2(G(k)\setminus G(\mathbb A))$. To show its cuspidality we use our assumption:
if $P$ is a $k$--parabolic subgroup of $G$ such that
a Levi subgroup of $P(k_v)$ contains a conjugate of a Levi subgroup defining  $\mathfrak M_v$ for all finite $v$ in
$S$,  then $P=G$, and  apply (\cite{Muic1}, Proposition 5.3). Thus, we obtain

$$
P(f)\in L^2_{cusp}(G(k)  \setminus G(\bbA)).
$$

Let $\mathfrak V$ be a closed subspace of  $L^2_{cusp}(G(k)  \setminus G(\bbA))$ generated by $P(f)$. It can be
decomposed into a direct sum of  irreducible unitary representations of  $G(\bbA)$ each
occurring with a finite multiplicity:
$$
\mathfrak V=\hat{\oplus}_j \mathfrak U_j, \ \ \text{each $\mathfrak U_j$ is closed and irreducible.}
$$
Let us write according to this decomposition
\begin{equation}\label{gl-2}
P(f)=\sum_j\psi_j, \ \ \psi_j\in \mathfrak U_j.
\end{equation}
Since $P(f)$ generates $\mathfrak V$, we must have
$$
\psi_j\neq 0, \ \ \text{for all $j$.}
$$
Also, since
$$
P(f)\not\in L^2_{cusp, \ \text{$(\psi, U)$--degenerate}}(G(k)\setminus G(\bbA)),
$$
there exists an index $i$ such that we have
$$
\mathfrak U_i \not\subset L^2_{cusp, \ \text{$(\psi, U)$--degenerate}}(G(k)\setminus G(\bbA)).
$$

From now on, we use arguments similar to those used in the proof of (\cite{Muic1}, Theorem 7.2). We just outline the argument.
It follows from (\ref{gl-2}) that the following inner product is not zero:
\begin{equation}\label{gl-4000}
\int_{G(k)\setminus G(\mathbb A)} P(f)(g)\overline{\psi_i(g)}dg= \int_{G(k)\setminus G(\mathbb A)} |\psi_i(g)|^2 dg>0.
\end{equation}
Since the space of cusp forms is dense in $\mathfrak U_i$ we can assume that $\psi_i$ is a
cusp form in above inequality. In particular, this means that
\begin{equation}\label{gl-3}
\psi_i\in C^\infty(G(k)\setminus G(\mathbb A)).
\end{equation}
The integral on the left--hand side in (\ref{gl-4000}) can be written as follows:
\begin{equation}\label{gl-4}
\int_{G(\mathbb A)} f(g)\overline{\psi_i(g)}dg=\int_{G(k)\setminus G(\mathbb A)} P(f)(g)\overline{\psi_i(g)}dg>0.
\end{equation}

Next, as it is well--known in the unitary theory, the space
$\overline{\mathfrak U}_i$ consisting of all $\overline{\psi}$, $\psi \in \mathfrak U_j$, is a contragredient representation of
$\mathfrak U_i$. Next, (\ref{gl-3}) and (\ref{gl-4}) tell us that $f$ acts non--trivially on $\overline{\mathfrak U}_i$.
If we write
$$
(\mathfrak U_i)_K=\pi_\infty^i\otimes_{v\in V_f} \pi_v^i,
$$
then
$$
(\overline{\mathfrak U}_i)_K=\wit{\pi}_\infty^i\otimes_{v\in V_f} \wit{\pi}_v^i,
$$
and
$$
\wit{\pi}_\infty^i(f_\infty)\otimes_{v\in V_f} \wit{\pi}_v^i(f_v)\neq 0.
$$
In particular, for each finite place $v$, we have
\begin{equation}\label{gl-5}
\wit{\pi}_v^i(f_v)\neq 0.
\end{equation}
Since, $f_v=1_{G(\calO_v)}$, for all $v\not\in S$, (\ref{gl-5}) implies that $\wit{\pi}_v^i$ and hence $\pi_v^i$ are unramified.
Also, since for finite $v\in S$, $f\in C_c^\infty(G(k_v))(\mathfrak M_v)$, (\ref{gl-5}) and (\cite{Muic1}, Lemma 5.2 (ii)) imply that
$\wit{\pi}_v^i$ belongs to the class $\wit{\mathfrak M}_v$. Hence, $\pi_v^i$ belongs to the class $\mathfrak M_v$.
Thus, if we let $\mathfrak U=\mathfrak U_i$, then  (i) and (ii) hold. Finally, (iii) holds by Lemma \ref{gl-0}.
\end{proof}

\vskip .2in
The following result we need in the proof of Theorem \ref{lnd-4}.

\begin{Cor}\label{gl-8}
Assume that $G$ is a semisimple  quasisplit algebraic group defined over a number field $k$. Let $U$ be the unipotent radical
of a Borel subgroup defined over $k$. Let $\psi: U(k)\setminus U(\mathbb A)\longrightarrow \mathbb C^\times$
be a nondegenerate  character. Assume that $v_0$ is a finite place of $k$ such that $G$ is unramified over $k_{v_0}$ and such that
there exists a $\psi_{v_0}$--generic supercuspidal representation of  $G(k_{v_0})$. Then, for any other finite place $v$,
any  Bernstein's class which satisfies  $\cInd_{U(k_v)}^{G(k_v)}(\psi_v)(\mathfrak M_v)\neq 0$ is $\psi_v$--generic.
\end{Cor}
\begin{proof} This corollary is a direct consequence of Lemma \ref{gl-1}. We just need to select $S$ large enough
  such that it contains both $v$ and $v_0$. For each finite place $w\in S$, $w\neq v, v_0$, let $\mathfrak M_w$ be
  a Bernstein's class such that  $\cInd_{U(k_w)}^{G(k_w)}(\psi_w)(\mathfrak M_w)\neq 0$ (at least one such class exists  by
  Bernstein's theory since  $\cInd_{U(k_w)}^{G(k_w)}(\psi_w)\neq 0$). 
\end{proof}

\vskip .2in
The following theorem is the main result of the present section and the paper: 

\begin{Thm}\label{gl-6}
Assume that $G$ is a semisimple  algebraic group defined over a number field $k$. Let $U$ be a unipotent $k$-subgroup.
Let $\psi: U(k)\setminus U(\mathbb A)\longrightarrow \mathbb C^\times$
be a (unitary) character.  Let $S$ be a finite set of places, containing $V_\infty$, large enough such that $G$
 and $\psi$ are  unramified for $v\not\in S$ (in particular, $\psi_v$ is trivial on $U(\mathcal O_v)$).
For each finite place $v\in S$,  let $\mathfrak M_v$ be  a   $(\psi_v, U(k_v))$--
generic Bernstein's class such that the following holds:
 if $P$ is a $k$--parabolic subgroup of $G$ such that
a Levi subgroup of $P(k_v)$ contains a conjugate of a Levi subgroup defining  $\mathfrak M_v$ for all finite $v$ in
$S$,  then $P=G$. Then, there exists an irreducible subspace in
$L^2_{cusp}(G(k)\setminus G(\bbA))$ which is $(\psi, U)$--generic such that its
$K$--finite vectors  $\pi_\infty\otimes_{v\in V_f} \pi_v$ satisfy the following:
\begin{itemize}
\item[(i)] $\pi_v$ is unramified for $v\not\in S$.
\item[(ii)] $\pi_v$ belongs to the class $\mathfrak M_v$ for all finite $v\in S$.
\item[(iii)] $\pi_v$ is $(\psi_v, U(k_v))$--generic for all finite $v$.
\end{itemize}
\end{Thm}
\begin{proof} By Lemma \ref{lnd-0a} (iii), for each finite $v\in S$, the class $\wit{\mathfrak M}_v$
  satisfies $\cInd_{U(k_v)}^{G(k_v)}(\overline{\psi}_v)(\wit{\mathfrak M}_v)\neq 0$.
Thus, by Lemma  \ref{gl-1}, there exists an irreducible subspace $\mathfrak U$ in
$L^2_{cusp}(G(k)\setminus G(\bbA))$ which is $(\overline{\psi}, U)$--generic such that its
$K$--finite vectors  $\rho_\infty\otimes_{v\in V_f} \rho_v$ satisfy the following:
\begin{itemize}
\item[(a)] $\rho_v$ is unramified for $v\not\in S$.
\item[(b)] $\rho_v$ belongs to the class $\wit{\mathfrak M}_v$ for all finite $v\in S$.
\item[(c)] $\rho_v$ is $(\overline{\psi_v}, U(k_v))$--generic for all finite $v$.
\end{itemize}
The contragredient representation of $\mathfrak U$ can be realized on the space of all functions $\overline{\varphi}$ where
$\varphi$ ranges over $\mathfrak U$. Then, by conjugating the Fourier coefficient of $\mathfrak U$, we see that the
contragredient is
$(\psi, U)$--generic.  Thus, if we let $\pi_\infty=\wit{\rho}_\infty$ and $\pi_v =\wit{\rho}_v$, for $v\in V_f$, then we get
(i) and (ii) from (a) and (b), respectively. Finally,
(iii) follows from Lemma \ref{gl-0} since contragredient is $(\psi, U)$--generic.
\end{proof}

\vskip .2in
The following corollary of Theorem \ref{gl-6}  is a generalization of similar
results of Henniart, Shahidi, and Vign\' eras (\cite{Hen}, \cite{vig}, \cite{Sh}, Proposition 5.1).
They considered the case of generic cusp forms  having only supercuspidal representations as ramified local components.
Those forms  have non--trivial Fourier coefficients
with respect to $(\psi, U)$ where $B=TU$ is a Borel subgroup defined over $k$ ($T$ is a maximal torus,
$U$ is the unipotent radical,
both defined over $k$) of   $G$ assumed to be quasi--split, and $\psi$ is generic in the sense
that it is not trivial when restricted to any root subgroup $U_\alpha(\mathbb A)$, where $\alpha$ is a simple root
corresponding to the choice of 
$B$. As usual we call such cuspidal forms $\psi$--generic cuspidal forms.

\begin{Cor}\label{gl-7}
Assume that $G$ is a semisimple quasisplit algebraic group defined over a number field $k$. Let $U$ be the unipotent radical
of a Borel subgroup defined over $k$. Let $\psi: U(k)\setminus U(\mathbb A)\longrightarrow \mathbb C^\times$
be a nondegenerate  character.  Let $S$ be a finite set of places, containing $V_\infty$, large enough such that $G$
 and $\psi$ are  unramified for $v\not\in S$ (in particular, $\psi_v$ is trivial on $U(\mathcal O_v)$).
 For each finite place $v\in S$,  let $[M_v, \rho_v]$ be  a Bernstein's class such that $M_v$ is a standard Levi subgroup of
 $G(k_v)$ and
 $\rho_v$ is a $\psi'_v$--generic supercuspidal representation of $M_v$ (see  the paragraph containing
  (\ref{lnd-2}) in Section \ref{lnd} for notation). Assume that the following holds:
if $P$ is a $k$--parabolic subgroup of $G$ such that
a Levi subgroup of $P(k_v)$ contains a conjugate of $M_v$ for all finite $v\in S$,  then $P=G$.
 Then, there exists an irreducible subspace in
$L^2_{cusp}(G(k)\setminus G(\bbA))$ which is $\psi$--generic such that its
$K$--finite vectors  $\pi_\infty\otimes_{v\in V_f} \pi_v$ satisfy the following:
\begin{itemize}
\item[(i)] $\pi_v$ is unramified for $v\not\in S$.
\item[(ii)] $\pi_v$ belongs to the class $[M_v, \rho_v]$ for all finite $v\in S$.
\item[(iii)] $\pi_v$ is $\psi_v$--generic for all finite $v$.
\end{itemize}
\end{Cor}


\section{Genericness of the representations of \cite{MoyMuic} }\label{sec-6}

\smallskip

Suppose $k_{v}$ is a p-adic field with ring of integers ${\mathcal R}_v$.  Let $G$ be a split simple algebraic group defined over ${\mathcal R}_v$.  
As in ({\cite{MoyMuic}, \S 3.2}), set 
$$
{\mathscr G} \ := \ G(k_{v}) \, , {\text{\rm{\ and \ }}}
{\mathcal K} := G({\mathcal R}_v) \ {\text{\rm{a maximal compact subgroup of ${\mathscr G}$.}}}
$$
\noindent If $L \subset G$ is a subgroup defined over ${\mathcal R}_{v}$, let $L_{v} = L(k_{v})$ be the group of $k_{v}$-rational points.   Let $B$ be a Borel subgroup defined 
over ${\mathcal R}_v$. 

\smallskip

Let $\ScptB ({\mathscr G})$ be the Bruhat-Tits building of ${\mathscr G}$.  Let $x_{\mathcal K} \in \ScptB ({\mathscr G})$ be the point fixed by ${\mathcal K}$.    The Borel subgroup $B$ then determines an Iwahori subgroup ${\mathcal I} \subset {\mathcal K}$.     Let $C = \ScptB ({\mathscr G})^{\mathcal I}$ be the fixed points of the Iwahori subgroup ${\mathcal I}$.  It is an alcove in $\ScptB ({\mathscr G})$.

\smallskip

Take a maximally split torus $A \subset B$ defined over ${\mathcal R}_{v}$ so that $C$ is contained in the apartment ${\ScptA} (A_v)$ associated to $A_v$.  Let $\Phi= \Phi(G,A)$ and $\Phi^{+}= \Phi^{+}(B,A)$ be the root system of $A$ and positive root system with respect to $G$ and $B$.  

\smallskip 

For $\alpha \in \Phi$, let $U_{\alpha} \subset G$ denote the corresponding root group.  We have

\begin{equation}
U (k_{v}) \ = \ {\underset {\alpha \in \Phi^{+}}{\prod}} U_{\alpha} (k_{v}) \, . \end{equation}

Let $\Gamma = {\mathbb Z} \, \gamma_{0} \, \subset {\mathbb Q}$ be the additive subgroup so that the affine roots have the form $\alpha + \eta$ with $\alpha \in \Phi$ and $\eta \in \Gamma$.  Let $U_{\alpha+\eta}$ be the subgroup of $U_{\alpha}(k_{v})$ associated to the affine root $\alpha + \eta$.

\smallskip

Let $\Delta$ and $\Delta^{\text{\rm{aff}}}$ be the simple roots and simple affine roots of $\ScptA (A_v)$ with respect to the Borel and Iwahori subgroups $B$ and ${\mathcal I}$ respectively.  We recall that every $\alpha \in \Delta$ is the gradient part of a unique root  $\psi \in \Delta^{\text{\rm{aff}}}$.  In this way, we view $\Delta$ as a subset of $\Delta^{\text{\rm{aff}}}$.

\smallskip

\smallskip

Let $\beta \in \Phi^{+}$ be the highest root, and let $-\beta + \gamma_{0}$ ($\gamma_{0} > 0$) be the simple affine root.  Let $\ell$ be the height of $\beta$ and take $x_{0} \in C$ to be the point satisfying 
\begin{equation}
\forall \ \alpha \in \Delta \subset \Delta^{\text{\rm{aff}}} \ : \quad \alpha (x_{0}) \ = \ -\beta (x_{0}) + \gamma_{0} \ = \ {\frac{\gamma_{0}}{\ell + 1}}
\end{equation}

\smallskip

\noindent For $j \ge 0$ an integer, set: 

\begin{equation}\label{moy-prasad-0}
j' \ := \ j+ ({\frac{\gamma_{0}}{\ell+1}}) \ .
\end{equation}

\noindent So,
\begin{equation}
\aligned
\forall \ \alpha \in \Delta  \ : \qquad \qquad ( \, \alpha \, + \, j \, )\, (x_{0}) \ &= \ j \ + \ {\frac{\gamma_{0}}{\ell + 1}} \ , \\
\ &\ \\
{\text{\rm{and}}} \quad (\, -\beta \, + \, \gamma_{0} \, + \, j \, ) \, (x_{0}) \ &= \ j \ + \ {\frac{\gamma_{0}}{\ell + 1}} \ .
\endaligned
\end{equation}

\smallskip

Let $\Phi^{\text{\rm{aff}}}$ denote the affine roots.  We consider the Moy-Prasad groups
\begin{equation}
{\mathscr G}_{{x_{0}},j'} \ = \ (A_{v})_{j} \, {\underset 
{\begin{matrix} \psi \in \Phi^{\text{\rm{aff}}} \\ 
\psi (x_{0}) \ge j' 
\end{matrix}
}
{\prod}} U_{\psi} 
\ \ {\text{\rm{ and }}} \ \ 
{\mathscr G}_{{x_{0}},(j')^{+}} \ = \ (A_{v})_{j} \, {\underset 
{\begin{matrix} \psi \in \Phi^{\text{\rm{aff}}} \\ 
\psi (x_{0}) > j' 
\end{matrix}
}
{\prod}} U_{\psi} 
\end{equation}

\noindent defined in \cite{MoyPr1, MoyPr2}, and 

\begin{equation}\label{quotient-0}
{\text{\rm{ the quotient group }}} \ \ {\mathscr G}_{{x_{0}},j'} / {\mathscr G}_{{x_{0}},(j')^{+}} \ \ {\text{\rm{ is canonically }}} {\underset {\psi \in \Delta^{\text{\rm{aff}}}}{\prod}} U_{(\psi+j)} / U_{(\psi+j^{+} )} \ .
\end{equation}

As in ({\cite{MoyMuic}, \S 3.2}), let $\chi$ be a character of the quotient 
${\mathscr G}_{{x_{0}},j'} / {\mathscr G}_{{x_{0}},(j')^{+}}$ which is non-degenerate in the sense that under the canonical 
isomorphism of \eqref{quotient-0}, $\chi$ is non-trivial on each of the groups $U_{(\psi+j)} / U_{(\psi+j^{+} )}$.  Then, the proof of Lemma 3-19 in 
({\cite{MoyMuic}, \S 3.2}) generalizes to show the following Lemma:

\smallskip

\begin{Lem} \label{cusp-lemma} Let $\chi$ be a non-degenerate character of  ${\mathscr G}_{{x_{0}},j'} / {\mathscr G}_{{x_{0}},j^{+}}$.  Then, 
\smallskip
\begin{itemize}
\item[(i)] The inflation of $\chi$ to ${\mathscr G}_{{x_{0}},j'}$, when extended to ${\mathscr G}$ by zero outside ${\mathscr G}_{{x_{0}},j'}$, is a cusp form of ${\mathscr G}$.  
\smallskip
\item[(ii)] For each $j \ge 0$, there exists an irreducible supercuspidal representation $(\rho, W)$ which has a non-zero ${\mathscr G}_{{x_{0}},(j')^{+}}$--invariant vector but no non-zero ${\mathscr G}_{{x_{0}},j'}$--invariant vector.
\end{itemize}
\end{Lem}

We show the irreducible supercuspidal representations arising from the cusp form $\chi$  are generic for a suitable (non-degenerate) character of the unipotent radical $U(k_{v})$ of $B(k_{v})$.

\smallskip

Recall the cusp form $\chi$ satisfies the following:  For $\alpha \in \Delta$ (positive simple roots), the restriction of the character $\chi$ to $U_{\alpha+j}$ factors to a non-trivial character of $U_{\alpha+j}/U_{\alpha+j^{+}}$.  Let $\xi$ be a character of $U(k_{v})$ so that:

\begin{equation}\label{character-0}
\xi_{| \, U_{\alpha+j}} \ {\text{\rm{ equals \ $\chi_{|\, U_{\alpha+j}}$ . }}}
\end{equation}

\noindent Clearly, $\xi$ is a non-degenerate character of the unipotent group $U(k_{v})$.

\smallskip

Recall for $f \in C^{\infty}_{c}(G(k_{v}))$, the Fourier coefficient of $f$ along $U(k_{v})$ with respect to $\xi$ is the function
$\mathcal F_{(\xi, U(k_{v}))}(f)$ on $G$ defined as:

\begin{equation}
\mathcal F_{(\xi, U(k_{v}))}(f) (g) \ := \ \int_{U(k_{v})} f(ug) \, \overline{\xi(u)} \, du \ .
\end{equation}

\noindent The coefficient $\mathcal F_{(\xi, U(k_{v}))}(f)$ lies the space:

\begin{equation}
{\text{\rm{c-Ind}}}^{G(k_{v})}_{U(k_{v})}(\xi) \ .
\end{equation}

\smallskip

\begin{Prop}\label{allen-1}
  Consider the cusp form $\chi$ defined in Lemma \eqref{cusp-lemma}, and $\xi$ a character of $U(k_{v})$ satisfying
  \eqref{character-0}.  Then, the Fourier coefficient $\mathcal F_{(\xi, U(k_{v}))}(\chi)$ satisfies,
  $\mathcal F_{(\xi, U(k_{v}))}(\chi)(1)$ is non-zero. 
\end{Prop}

\begin{proof}

\begin{equation}
\aligned
\mathcal F_{(\xi, U(k_{v}))}(\chi)(1) \ &= \ \int_{ U \, \cap \, {\mathscr G}_{{x_{0}},j'} } \chi(u) \ \overline{\xi (u)} \ du \
= \ \int_{ U \, \cap \, {\mathscr G}_{{x_{0}},j'}} 1 \ du \\
&= {\text{\rm{meas}}} (  U \, \cap \, {\mathscr G}_{{x_{0}},j'} )
\endaligned
\end{equation}

\noindent In particular, the Fourier coefficient function $\mathcal F_{(\xi, U(k_{v}))}(\chi)$ is a non-zero function.

\end{proof}

\medskip

Let $V_{\chi}$ be the $G(k_{v})$-subrepresentation of $C^{\infty}_{c}(G(k_{v}))$ consisting of the right translates of $\chi$.
It is a finite length supercuspidal representation of $G(k_{v})$, and

\begin{equation}
\mathcal F_{(\xi, U(k_{v}))} \ : \ V_{\chi} \ \longrightarrow \ {\text{\rm{c-Ind}}}^{G(k_{v})}_{U(k_{v})}(\xi)
\end{equation}

\noindent is a $G(k_{v})$-map.  Let $\mathfrak B$ be the finite number of Bernstein components which appear in $V_{\chi}$.
The Bernstein projection of $V_{\chi}$ to itself according to the components in $\mathfrak B$.  Similarly, let
${\text{\rm{c-Ind}}}^{G}_{U}(\xi) ({\mathfrak B})$ be the Bernstein projection of
${\text{\rm{c-Ind}}}^{G}_{U}(\xi) ({\mathfrak B})$ to the $\mathfrak B$ components.   Then

\begin{equation}
\mathcal F_{(\xi, U(k_{v}))} \ : \ V_{\chi} \ \longrightarrow \ {\text{\rm{c-Ind}}}^{G(k_{v})}_{U(k_{v})}(\xi) ({\mathfrak B}) \ ,
\end{equation}

\noindent and the non-zero Fourier coefficient function $\mathcal F_{(\xi, U(k_{v}))}(\chi )$ belongs to
${\text{\rm{c-Ind}}}^{G(k_{v})}_{U(k_{v})}(\xi) ({\mathfrak B})$.

\section{A Relation to \cite{MoyMuic}}\label{newforms}

In this section we combine the results of current paper with the results of our previous paper \cite{MoyMuic} in order
to prove the existence of generic  cuspidal forms on a simply connected absolutely almost simple algebraic group $G$ 
defined over $\bbQ$ such that $G_\infty = G(\bbR )$ is not compact. We remind the reader that these are the assumptions
of \cite{MoyMuic}.
Examples of such groups are split Chevalley groups such as $SL(n)$, $Sp(n)$, or split $G_2$. In this section we
let $k=\mathbb Q$.

For each prime $p$, let $\bbZ_p$ denote the p-adic integers inside $\bbQ_p$.  Recall that for almost all primes $p$,
the group $G$ is 
unramified over $\bbQ_p$.  Thus, $G$ is a group scheme over $\bbZ_p$, and   $G(\bbZ_p)$ is a hyperspecial maximal compact
subgroup of  
$G(\bbQ_p)$ (\cite{Tits}, 3.9.1).

As in Section \ref{fc}, we let $U$ be a unipotent $\mathbb Q$-subgroup of $G$. Let $\psi: U(\mathbb Q)\setminus U(\mathbb A)
\longrightarrow \mathbb C^\times$ be a (unitary) character.

As in (\cite{MoyMuic}, Assumptions 1-3) we consider a finite family of open compact subgroups but which satisfy more
restrictive properties.  We consider a finite family of open compact subgroups
\begin{equation}\label{finite-family}
{\mathcal F} = \{ L \}
\end{equation}
satisfying the following assumptions:

\medskip
\begin{Assumptions}\label{assumptions} \ 

\smallskip
\begin{itemize}
\item[(i)] Under the partial ordering of inclusion there exists a subgroup 
$L_{\text{\rm{min}}} \in {\mathcal F}$ that is a subgroup of all the others.  
\smallskip 
\item[(ii)] The groups $L\in {\mathcal F}$ are factorizable, i.e., $L={\underset p \prod} L_p$, and for all but finitely
  many $p$'s, the group $L_p$ is the maximal compact subgroup $G(\mathbb Z_p)$.  
\smallskip 
\item[(iii)]  There exists a non-empty finite set of primes $T$ such that for $p\in T$  the group $G(\mathbb Q_p)$ has a
  local cusp form $f_p\in C_c^\infty(G(\mathbb Q_p))$ which satisfies the following conditions:
\begin{itemize}
\item[(a)]  $f_p$ is $L_{\text{\rm{min}}, p}$--invariant on the right, and 
\item[(b)] $\int_{U(\mathbb Q_p)} f_p(u_p)\overline{\psi_p(u_p)} du_p\neq 0$.
\end{itemize}
Moreover, we assume that for $L\neq L_{\text{\rm{min}}}$ there exists $p\in T$ such that the integral $\int_{L_p} f_p(g_pl_p) dl_p=0$
for all $g_p\in G(\mathbb Q_p)$.
\item[(iv)] $\psi_p$ is trivial on $U(\mathbb Q_p)\cap L_{\text{\rm{min}}, p}$ for all $p\not\in T$.
\end{itemize}
\end{Assumptions}

The reader may want to compare these assumptions with  (\cite{MoyMuic}, Assumptions 1-3). We remark that using results of 
Section \ref{sec-6} we can  write down examples of families $\mathcal F$ satisfying 
Assumptions \ref{assumptions} in the ordinary generic case (see Introduction) by globalizing non--degenerate characters from that section. 
But this is very technical and we do not write down details here. 
Analogous result can be found in  \cite{MoyMuic}.

Let $L\subset  G(\mathbb A_f)$ be an open compact subgroup. We define a congruence subgroup $\Gamma_L$ of $G_\infty$
using (\ref{discreteS-01}).
We define $L^2_{cusp}(\Gamma_{L} \backslash G_\infty)$ to be the subset of $L^2(\Gamma_{L} \backslash G_\infty)$ consisting of all 
measurable functions $\varphi \in L^2(\Gamma_{L} \backslash G_\infty)$ such that $$
\int_{U_P(\mathbb R)\cap \Gamma_{L}\backslash U_P(\mathbb R)}\varphi (ug)=0, \ \ \text{(a.e.) for $g\in G_\infty$,}
$$
where $U_P$ is the unipotent radical of any proper $\mathbb Q$--parabolic subgroup $P$.

Further, assume that $L$ is factorizable $L=\prod_p L_p$ and that $\psi_p$ is trivial on $L_p\cap U(\mathbb Q_p)$ for all $p$.
Then, $\psi_\infty$ is trivial on $U_\infty \cap \Gamma_L$.
We remind the reader that in the proof of Lemma \ref{fc-1} we proved that $U_\infty \cap \Gamma_L\setminus U_\infty$ is
compact.   The basic considerations similar to those given at the beginning of Section \ref{fc} can be carried without
difficulties. So, as in Section \ref{fc}, for $\varphi \in L^2(\Gamma_{L} \backslash G_\infty)$, 
we define the $(\psi_\infty, U_\infty)$--Fourier coefficient
$$
\mathcal F_{(\psi_\infty, U_\infty)}(\varphi)(g_\infty)=
\int_{U_\infty \cap \Gamma_L\setminus U_\infty}\varphi(u_\infty g_\infty)\overline{\psi_\infty(u_\infty)} du_\infty, \ \ \text{a.e. for $g_\infty\in G_\infty$.}
$$
We say that $\varphi$ is
$(\psi_\infty, U_\infty)$--generic if $\mathcal F_{(\psi_\infty, U_\infty)}(\varphi)\neq 0$ (a.e.).  We define
the closed $G_\infty$--invariant subspace $L^2_{ \text{$(\psi_\infty, U_\infty)$--degenerate}}(\Gamma_L\setminus G_\infty)$ as in Section \ref{fc}.
As in Section \ref{gl}, we define 
$$
L^2_{cusp, \ \text{$(\psi_\infty, U_\infty)$--degenerate}}(\Gamma_L\setminus G_\infty)=
L^2_{ \text{$(\psi_\infty, U_\infty)$--degenerate}}(\Gamma_L\setminus G_\infty)\cap
L^2_{cusp}(\Gamma_{L} \backslash G_\infty).
$$
As before, we say that an irreducible closed subrepresentation in $L^2_{cusp}(\Gamma_{L} \backslash G_\infty)$ is
$(\psi_\infty, U_\infty)$--generic if
$$
\mathfrak U\not\subset L^2_{cusp, \ \text{$(\psi_\infty, U_\infty)$--degenerate}}(\Gamma_L\setminus G_\infty).
$$
As in the proof of Lemma \ref{gl-0}, we see that the functional from the space of cuspidal automorphic forms (i.e., the 
space of $K_\infty$--finite vectors) in $\mathfrak U$ given by
$$
\varphi\longmapsto \int_{U_\infty \cap \Gamma_L\setminus U_\infty}\varphi(u_\infty)\overline{\psi_\infty(u_\infty)} du_\infty
$$
is not zero.

After these preliminaries, we are ready to state and prove the main result of the present section. It is analogous to the
main result of \cite{MoyMuic}.

\begin{Thm}\label{intr-thm}  Suppose $G$ is a simply connected, absolutely almost simple algebraic group defined over
  ${\mathbb Q}$, such that $G_{\infty}$ is
  non-compact and ${\mathcal F} = \{ L \}$ is a finite set of open compact subgroups of $G({\bbA}_{f})$ satisfying assumptions
  \eqref{assumptions}. 
 Then, the orthogonal complement of 
$$
\sum_{\substack{ L \in {\mathcal F}  \\ L_{\text{\rm{min}}} \subsetneq L }} L^2_{cusp}(\Gamma_L \backslash G_\infty)
$$ 
in 
$L^2_{cusp}(\Gamma_{L_{\text{\rm{min}}}} \backslash G_\infty)$ contains an $(\psi_\infty, U_\infty)$--generic irreducible
(closed) subrepresentation.
\end{Thm}
\begin{proof} The proof of this theorem  is similar to the proof of (\cite{MoyMuic}, Theorem 1-4) but instead of
  (\cite{Muic1}, Theorem 4-2), we use Lemma  \ref{fc-1}. For $p\not\in T$, we let $f_p=1_{ L_{\text{\rm{min}}, p}}$. For $p\in T$, we use the cusp form $f_p$ given by  
Assumption \ref{assumptions} (iii).

Now, in view of our Assumptions \ref{assumptions}, we see that 
all assumptions (a)--(c)  of Lemma \ref{fc-1} hold. As a consequence, Lemma \ref{fc-1}
asserts that there exists   $f_\infty \in
C_c^\infty(G_\infty)$, $f_\infty \neq 0$, such that if we let  $f=f_\infty\otimes_{p} f_p$, then  the following holds:
\begin{equation}\label{aaa}
\int_{ U_\infty \cap \Gamma_{L_{\text{\rm{min}}}} \setminus U_\infty}P(f)(u_\infty)\overline{\psi_\infty(u_\infty)}du_\infty\neq 0. 
\end{equation}
Next, as in the proof of Lemma \ref{gl-1}, we see that $P(f)$ is cuspidal. Hence, (\cite{Muic1}, Proposition 3.2) 
implies that $P(f)|_{G_\infty}$ is $\Gamma_L$--cuspidal. Thus, (\ref{aaa}) implies that $P(f)|_{G_\infty}$ is a non--zero element of 
$L^2_{cusp}(\Gamma_{L_{\text{\rm{min}}}}\setminus G_\infty)$.

Next, as in (\cite{MoyMuic}, Lemmas 2-18, 2-19), we show that $P(f)|_{G_\infty}$ is orthogonal to 
$L^2_{cusp}(\Gamma_{L}\setminus G_\infty)$ in  $L^2_{cusp}(\Gamma_{L_{\text{\rm{min}}}}\setminus G_\infty)$
for all $L\in \mathcal F$, $L\neq L_{\text{\rm{min}}}$. Thus, the closed  $G_\infty$--invariant subspace $\mathcal U$ in 
$L^2_{cusp}(\Gamma_{L_{\text{\rm{min}}}}\setminus G_\infty)$ generated by $P(f)|_{G_\infty}$ is non--trivial by (\ref{aaa}), and
consequently direct sum of 
irreducible unitary representations each appearing with finite multiplicity \cite{godement}. Finally, using (\ref{aaa}) and arguing as 
in the proof of Lemma \ref{gl-1}, we see that  some of those representations must be 
$(\psi_\infty, U_\infty)$--generic.
\end{proof}

\end{document}